\documentclass{article}
\usepackage{amsmath}
\usepackage{amsthm}
\usepackage{mathrsfs}
\usepackage{latexsym}
\usepackage{amssymb}
\usepackage{amscd}
\usepackage[all,cmtip]{xy}

\DeclareSymbolFont{AMSb}{U}{msb}{m}{n}
\DeclareMathSymbol{\N}{\mathbin}{AMSb}{"4E}
\DeclareMathSymbol{\Z}{\mathbin}{AMSb}{"5A}
\DeclareMathSymbol{\R}{\mathbin}{AMSb}{"52}
\DeclareMathSymbol{\Q}{\mathbin}{AMSb}{"51}
\DeclareMathSymbol{\I}{\mathbin}{AMSb}{"49}
\DeclareMathSymbol{\C}{\mathbin}{AMSb}{"43}

\newcommand{\dbl}{[\hspace{-0.2ex}[}
\newcommand{\dbr}{]\hspace{-0.2ex}]}
\newcommand{\db}[1]{\dbl {#1} \dbr}
\newcommand{\iso}{\cong}
\newcommand{\invlim}{\underleftarrow{\textnormal{lim}}\,}
\newcommand{\dirlim}{\underrightarrow{\textnormal{lim}}\,}
\newcommand{\ctens}{\widehat{\otimes}}

\newcommand{\onto}{\twoheadrightarrow}

\newcommand{\Hom}{\textnormal{Hom}}

\newcommand{\rad}{\textnormal{rad}}

\newcommand{\SL}{\textnormal{SL}}
\newcommand{\dash}{\textnormal{-}}
\newcommand{\tn}[1]{\textnormal{#1}}
\newcommand{\cat}[1]{\tn{\textbf{#1}}}
\newcommand{\F}{\mathbb{F}}
\newcommand{\bigO}{\mathcal{O}}
\newcommand{\dbsub}[1]{_{{}_{#1}}}

\numberwithin{equation}{section}

\begin{document}

\author{John MacQuarrie\footnote{Research supported by a Heilbronn Research Fellowship.}, Peter Symonds}

\title{Brauer Theory for Profinite Groups}


\maketitle

\newtheorem{defn}[equation]{Def{i}nition}
\newtheorem{prop}[equation]{Proposition}
\newtheorem{lemma}[equation]{Lemma}
\newtheorem{theorem}[equation]{Theorem}
\newtheorem{corol}[equation]{Corollary}
\newtheorem{question}[equation]{Questions}

\maketitle

\section{Introduction}

Let $G$ be a finite group and $(K,\bigO,k)$ a $p$-modular system.  The subject known today as Brauer Theory is concerned with modular representations of $G$ and how they relate to ordinary representations.  Since the principal concern is with composition factors of a module, the language of Grothendieck groups has proved valuable.  Here, we discuss in a functorial manner what Serre \cite{serreRepsFiniteGroups} calls ``the $cde$ triangle''
$$
\xymatrix{
P_k(G) \ar[dr]_{e} \ar[rr]^{c} && R_k(G)  \\
& R_K(G)\ar[ur]_d &}
$$
wherein $R_K(G), R_k(G)$ are the Grothendieck groups of finitely generated $KG$-modules and $kG$-modules respectively, $P_k(G)$ is the Grothendieck group of finitely generated projective $kG$-modules, $c$ is the Cartan homomorphism and $d$ is the decomposition map.

\medskip

For appropriate coefficient rings $A$, we will note that $R_A(-)$ and $P_A(-)$ are functors.  From this perspective, $d$ is a natural transformation from $R_K(-)$ to $R_k(-)$.  While $e$ is a natural map from $P_{\bigO}(-)$ to $P_K(-)$, it cannot naturally be thought to have codomain $R_K(-)$, and hence the maps $e$ and $d$ are not naturally composable.  Using the functors and natural transformations discussed above, we extend $d$ and $e$ to profinite groups.

In Section \ref{SECTION characters} we discuss characters for profinite groups, and in Section \ref{SECTION blocks} we note that as with finite groups, the representation theory of $A\db{G}$ splits up into blocks.  Finally, in Section \ref{SECTION Cartan} we give a method that computes the Cartan matrix of a finite group $G$ in terms of the Cartan matrix of $G/U$ for a normal $p$-subgroup $U$ of $G$.  This allows one to give a closed formula for the Cartan matrix of groups in a cofinal system of finite quotients of an analytic pro-$p$ group.

\section{Preliminaries}

We collect here some basic results about $p$-modular systems and the module categories that will interest us.

\subsection{$p$-modular systems}

Fix once and for all a prime $p$ and an algebraic closure $\overline{\Q}_p$ of the field $\Q_p$ of $p$-adic rationals.  With respect to this closure, we define the category of $p$-modular systems as follows.  The objects are triples $(K,\bigO,k)$, where $K$ is a finite subextension of $\Q_p$, $\bigO$ is the valuation ring of $K$ (a finite extension of $\Z_p$), and $k$ is the residue field of $\bigO$ (a finite extension of $\F_p$).  Given objects $(K,\bigO,k)$ and $(K',\bigO',k')$, an inclusion $K\hookrightarrow K'$ induces maps $\bigO\to \bigO'$ and $k\to k'$, and the collection of these triples of maps form the morphisms of our category.  With respect to the obvious ordering $(K,\bigO,k)\leqslant (K',\bigO',k)$ whenever $K\leqslant K'$, this category forms an upwards directed set.  A $p$-modular system $(K,\bigO,k)$ is completely determined by $K$, and we frequently use this observation to suppress notation.

Given a finite group $G$, recall that a field $A$ is said to be a \emph{splitting field} for $G$ if every irreducible $AG$-module is absolutely irreducible.  We can always find a finite extension $K$ of $\Q_p$ such that $K$ is a splitting field for $G$ and all its subgroups, in which case $K$ is sometimes said to be ``sufficiently large'' for $G$.   If $K$ is sufficiently large for $G$, then so is $k$.  For proofs of all these facts, see \cite[\S 17A]{curtisandreiner1}.

\subsection{The rings $k\db{G}, \bigO\db{G}, K\db{G}$ and their modules}

Let $G$ be a profinite group.  The coefficient rings $k$ and $\bigO$ (being finite extensions of $\F_p, \Z_p$ respectively) are profinite, so that for $A\in \{\bigO,k\}$ the algebra
$$A\db{G} := \invlim_{N\lhd_O G}A[G/N]$$
is profinite (see \cite[\S 5.3]{ribzal} for more details).  We are interested in the category $A\db{G}\dash\cat{mod}$ of finitely generated profinite left $A\db{G}$-modules.  The category $k\db{G}\dash\cat{mod}$ is the object of study in the modular representation theory of profinite groups, considered in \cite{macquarrierelproj} and \cite{macquarrieGreenCorrespondence}. Both categories have enough projectives.

\medskip

It is not the case that $K\otimes_{\bigO}\bigO\db{G} \iso \invlim_{N\lhd_O G}\, K\otimes_{\bigO}\bigO[G/N]$.  We will work with the latter, so define
$$K\db{G}:=\invlim_{N\lhd_O G}\,K[G/N].$$

Denote by $K\db{G}\dash\cat{mod}$ the category of Artin $K\db{G}$-modules.  That is, the category of finite dimensional $K\db{G}$-modules having an open normal subgroup of $G$ in the kernel of their action.  Indecomposable objects in this category are irreducible Artin modules.

\section{Grothendieck groups}\label{SECTION grothendieck groups}

In this section we first give a careful analysis of various Grothendieck groups associated to a finite group $G$, describing functorial relationships between them.  We then describe a natural generalization to the profinite setting.

\subsection{Grothendieck groups of finite groups}

Let $(K,\bigO,k)$ be a $p$-modular system as above.  For a finite group $G$ and $A\in \{K,k\}$, let $R_A(G)$ denote the Grothendieck group of finitely generated $AG$-modules.  By the Krull-Schmidt Theorem, this is the free abelian group with basis the isomorphism classes of irreducible $AG$-modules.  Denote the image of an $AG$-module $X$ in $R_A(G)$ by $[X]$.  Given a group homomorphism $\alpha:G\to H$, we define a group homomorphism $R_A(\alpha):R_A(H)\to R_A(G)$ on an $AH$-module $V$ by $[V]\mapsto[{}^{\alpha}V]$, where ${}^{\alpha}V$ is the vector space $V$ with action from $G$ given by $g\cdot v:=\alpha(g)v$ (``the module $V$ restricted to $G$'').  In this way, we regard $R_A(-)$ as a contravariant functor from finite groups to abelian groups.

For $A\in \{K,\bigO,k\}$, let $P_A(G)$ denote the Grothendieck group of finitely generated projective $AG$-modules.  Again, we are entitled to think of $P_A(G)$ as the free abelian group with basis the isomorphism classes of indecomposable projective $AG$-modules.  We have an isomorphism of groups $P_K(G)\iso R_K(G)$, but the functors will be rather different.  Given a group homomorphism $\alpha:G\to H$, define $P_A(\alpha):P_A(G)\to P_A(H)$ on a projective $AG$-module $P$ by $[P]\mapsto [AH^{\alpha}\otimes_{AG}P]$, where $AH^{\alpha}$ is the $(AH\dash AG)$-bimodule whose action from $G$ is given by $x\cdot g:=x\alpha(g)$.  In this way, we regard $P_A(-)$ as a covariant functor from finite groups to abelian groups.

\medskip

Just as above, we can define functors $R_{\overline{\Q}_p}(-), R_{\overline{\F}_p}(-), P_{\overline{\Q}_p}(-)$ and $P_{\overline{\F}_p}(-)$.  For $A$ a field, $R_A(-)$ and $P_A(-)$ commute with scalar extension (as one can see by doing simple checks and using results in \cite[\S 14.6]{serreRepsFiniteGroups}), and hence the functors defined over algebraic closures are equal to the direct limits of the corresponding functors as the coefficient rings vary over finite extensions of the ground field.  The ring $\dirlim \bigO$ is not a discrete valuation ring, so to avoid unnecessary technicalities with modules over $\overline{\bigO}$, we simply \emph{define}
$$P_{\overline{\bigO}}(-):=\dirlim P_{\bigO}(-)$$
as $\bigO$ ranges through the valuation rings of finite extensions of $\Q_p$ in the obvious way.

\medskip

Let $F$ be a functor from a category $\mathcal{C}$ to topological abelian groups.  We denote by $F(-)^{*}$ the composition of $F$ with the functor $\Hom_{\Z}(-,\Z)$ from topological abelian groups to itself that sends an object $X$ to the group of continuous homomorphisms from $X$ to the discrete group $\Z$.  The topology on $X^*$ is the compact-open topology.  For the functors $F$ that will concern us, $(F(-)^*)^*$ will be naturally isomorphic to $F$, so that $(-)^*$ is a duality.

As in \cite[\S 14.5]{serreRepsFiniteGroups} for $A\in \{K,k\}$ we define the pairing
$$\langle -,-\rangle_G :P_A(G)\times R_A(G)\to \Z$$
on modules by
$$\langle [P],[V]\rangle_G := \dim_A \Hom_{AG}(P,V)$$
and extending bilinearly.

\begin{prop}
For $A\in \{K,k\}$ the maps $\eta\dbsub{A,G}:P_A(G)\to R_A(G)^*$ given by $P\mapsto \langle P,-\rangle\dbsub{G}$ are the components of a natural transformation $\eta\dbsub{A}:P_A(-)\to R_A(-)^*$. The component $\eta\dbsub{A,G}$ is an isomorphism if and only if $A$ is a splitting field for $G$.
\end{prop}

\begin{proof}
That  $\eta_A$ is natural amounts to the assertion that for $\alpha:G\to H$ a group homomorphism, $P$ a projective $AG$-module and $V$ an $AH$-module, we have an isomorphism of vector spaces
$$\Hom_{AH}(AH^{\alpha}\otimes_{AG}P, V)\iso \Hom_{AG}(P,{}^{\alpha}V).$$
We have a natural isomorphism $^{\alpha}V\iso \Hom_{AH}(AH^{\alpha},V)$, so the required isomorphism is a tensor-hom adjunction (for which, see eg. \cite[2.8.2]{benson}).

Each $\eta\dbsub{A,G}$ is injective, but maps onto the standard basis element $[\rho_X]\in R_A(G)^*$ ($\rho_X$ taking the irreducible $X$ to 1 and other irreducibles to 0) if and only if $X$ is absolutely irreducible, by \cite[29.13]{curtisandreinerOld}.
\end{proof}

Since the functors $P_A(-), R_A(-)$ commute with field extensions, we have isomorphisms of functors $P_{\overline{\F}_p}(-)\iso R_{\overline{\F}_p}(-)^*$ and $P_{\overline{\Q}_p}(-)\iso R_{\overline{\Q}_p}(-)^*$.  Note that as no finite field extension of the base-field is a splitting field for every group, we do not get functor isomorphisms for finite extensions of $\F_p$ or $\Q_p$.

\medskip

For any $p$-modular system $(K,\bigO,k)$ the functors $P_{\bigO}(-)$ and $P_k(-)$ are naturally isomorphic via $P\mapsto P/\mathfrak{p}P$, where $\mathfrak{p}$ is the maximal ideal of $\bigO$. We define a pairing $P_{\bigO}(-)\times R_{k}(-)\to \Z$ by sending the pair $(P,S)$ to $\dim_k\Hom_{kG}(P/\mathfrak{p}P,S)$.  As above, we see that the functors $P_{\overline{\bigO}}(-)$ and $R_{\overline{\F}_p}(-)^*$ are isomorphic.

\subsection{Grothendieck groups of profinite groups}

Now let $G= \invlim\dbsub{N\lhd_O G}G/N$ be a profinite group and let $(K,\bigO,k)$ be a $p$-modular system.  For $A\in \{K,k\}$, $R_A(-)$ is a contravariant functor, so the inverse system for $G$ induces a direct system of Grothendieck groups, and we define
$$\widehat{R}_A(G):=\dirlim_N R_A(G/N).$$
Similarly, for $A\in \{K,\bigO,k\}$, since $P_A(-)$ is a covariant functor we define
$$\widehat{P}_A(G):=\invlim_N P_A(G/N).$$
In this way we regard $\widehat{R}_A(-)$ as a contravariant functor from profinite groups to (discrete) topological abelian groups, and $\widehat{P}_A(-)$ as a covariant functor from profinite groups to topological abelian groups.

For $A\in \{\overline{\Q}_p,\overline{\F}_p\}$ the natural isomorphism $P_A(-)\iso R_A(-)^*$ from the previous section extends to a natural isomorphism $\widehat{P}_A(-)\iso \widehat{R}_A(-)^*$.  This isomorphism can be written explicitly in terms of pairings as follows.  Given a profinite group $G$ and modules $[P]\in \widehat{P}_A(G)$ and $[X]\in \widehat{R}_A(G)$, there is some $N_0\lhd_O G$ such that $[X]$ is the image of an element of $R_A(G/N_0)$ by basic properties of the direct limit.  Denoting by $[P_N]$ the image of $[P]$ in $P_A(G/N)$ we have that $\langle [P_N],[X]\rangle_{G/N}$ is independent of $N$ as $N$ ranges through those open normal subgroups of $G$ contained inside $N_0$.  We thus define
$$\langle [P],[X]\rangle_G:= \langle [P_{N_0}],[X]\rangle_{G/N_0}.$$
This pairing induces the required isomorphism $\widehat{P}_A(-)\iso \widehat{R}_A(-)^*$.  Similarly we can define a pairing that induces the natural isomorphism  $\widehat{P}_{\overline{\bigO}}(-)\iso \widehat{R}_{\overline{\F}_p}(-)^*$.

We give a more explicit description of the groups involved, as follows.

\begin{lemma}\label{grothendieck isomorphisms}
Let $A$ be $K$ or $k$.  We have isomorphisms
$$\widehat{R}_A(G)\iso \bigoplus_{V}\Z_{[V]}$$
and
$$\widehat{R}_A(G)^*\iso \prod_{V}\Z_{[V]}$$
where $V$ runs through the set of isomorphism classes of irreducible Artin $A\db{G}$-modules.
\end{lemma}

\begin{proof}
The first isomorphism is easy and the second follows from the first by duality.
\end{proof}

For $A\in \{\overline{\Q}_p, \overline{\F}_p\}$, the second isomorphism of Lemma \ref{grothendieck isomorphisms} and the functor isomorphisms above give the description of $\widehat{P}_A(G)$ as $\prod_{P}\Z_{[P]}$, where $P$ runs through the set of isomorphism classes of indecomposable projective $A\db{G}$-modules.  Note in particular that we do not have an isomorphism of groups between $\widehat{P}_K(G)$ and $\widehat{R}_K(G)$ when $G$ is an infinite profinite group.  We can express $\widehat{P}_{\overline{\bigO}}(G)$ as a product of copies of $\Z$ indexed by irreducible $\overline{\F}_p$-modules, by duality.

\section{The decomposition map and its dual}

We turn now to natural transformations between the functors discussed above.  Consider a finite group $G$ and fix a $p$-modular system $(K,\bigO,k)$.   Recall that the \emph{decomposition} map $d_{K,G}:R_K(G)\to R_k(G)$ is defined on the basis of $R_K(G)$ as follows.  Given an irreducible $KG$-module $V$, choose an $\bigO$-form $M$ for $V$ (that is, an $\bigO G$-lattice $M$ such that $K\otimes_{\bigO}M\iso V$).  Define $d_{K,G}([V]):=[M/\mathfrak{p}M]$, where $\mathfrak{p}$ is the maximal ideal of $\bigO$.  Although there may be non-isomorphic $\bigO$-forms $M$ of $V$, $[M/\mathfrak{p}M]$ is well-defined \cite[\S15.2]{serreRepsFiniteGroups}.

We also define a map $e_{K,G}:P_{\bigO}(G)\to P_K(G)$ by $e_{K,G}([P])=[K\otimes_{\bigO}P]$.

\begin{prop}
Let $(K,\bigO,k)$ be a $p$-modular system.  The map $d_{K,G}:R_K(G)\to R_k(G)$ given above is the component at $G$ of a natural transformation $d_K:R_K(-)\to R_k(-)$.  The map $e_{K,G}:P_{\bigO}(G)\to P_{K}(G)$ is the component at $G$ of a natural transformation $e_K:P_{\bigO}(-)\to P_K(-)$.
\end{prop}

\begin{proof}
Let $\alpha:G\to H$ be a group homomorphism.  After writing out the square, the first claim amounts to the assertion that
$$[M_{{}^{\alpha}V}/\mathfrak{p}M_{{}^{\alpha}V}] = [{}^{\alpha}(M_V/\mathfrak{p}M_V)],$$
where $M_X$ denotes an $\bigO$-form for $X$.  This is easily checked, making use of the exactness of ${}^{\alpha}(-)$.

The second claim is the assertion that for a projective $\bigO G$-module $P$,
$$K\otimes_{\bigO}\bigO H^{\alpha}\otimes_{\bigO G}P \iso KH^{\alpha}\otimes_{KG}K\otimes_{\bigO}P,$$
which is also easily checked.
\end{proof}

Note the following important remark.  For a finite group $G$, $e_{K,G}$ is usually thought to have codomain $R_K(G)$ (as in the diagram in the introduction).  This allows for the composition $d\circ e: P_k(G)\to R_K(G)\to R_k(G)$ known as the Cartan homomorphism.  However, even for finite groups $e_K$ is not a natural transformation between the functors $P_{\bigO}(-)$ and $R_K(-)$, but only between $P_{\bigO}(-)$ and $R_K(-)^*$.  It follows that $e$ and $d$ are not naturally composable.

\medskip

Now let $G = \invlim\dbsub{N\lhd_O G}G/N$ be a profinite group.  The following is immediate from the naturality of $d_K$ and $e_K$.

\begin{prop}
The maps $d_{K,G/N}$ define a map of direct systems
$$\{R_K(G/N)\,|\,N\lhd_O G\}\to \{R_k(G/N)\,|\,N\lhd_O G\},$$
and hence a map $\widehat{d}_{K,G}:\widehat{R}_K(G)\to \widehat{R}_k(G)$.  This map is the component at $G$ of a natural transformation $\widehat{d}_K:\widehat{R}_K(-)\to \widehat{R}_k(-)$.

\medskip

Likewise, the maps $e_{K,G/N}$ define a map of inverse systems
and hence a map $\widehat{e}_{K,G}:\widehat{P}_{\bigO}(G)\to \widehat{P}_K(G)$.  This map is the component at $G$ of a natural transformation $\widehat{e}_K:\widehat{P}_{\bigO}(-)\to \widehat{P}_K(-)$.
\end{prop}

The natural transformation $\widehat{d}_{\overline{\Q}_p}:\widehat{R}_{\overline{\Q}_p}(-)\to \widehat{R}_{\overline{\F}_p}(-)$ induces a natural transformation $\widehat{d}_{\overline{\Q}_p}^*:\widehat{R}_{\overline{\F}_p}(-)^*\to \widehat{R}_{\overline{\Q}_p}(-)^*$, which can naturally be thought of as a natural transformation $\widehat{P}_{\overline{\bigO}}(-)\to \widehat{P}_{\overline{\Q}_p}(-)$ via the isomorphisms in Section \ref{SECTION grothendieck groups}.

\begin{theorem}
The natural transformations $\widehat{e}_K$ and $\widehat{d}_K^*$ commute with the natural maps $\widehat{P}_{\bigO}(-)\to \widehat{R}_k^*(-)$ and $\widehat{P}_{K}(-)\to \widehat{R}_K^*(-)$ in the sense that the following squares commute:
$$
\xymatrix{
\widehat{P}_{\bigO}(G) \ar[r]^{\widehat{e}_{K,G}}\ar[d]     & \widehat{P}_K(G)\ar[d] \\
\widehat{R}_k(G)^* \ar[r]_{\widehat{d}_{K,G}^*}             & \widehat{R}_K(G)^*.}
$$
\end{theorem}

\begin{proof}
This is immediate from the finite theory.  For a finite group $G$ we are asserting that given a projective $\bigO G$-module $P$ and finitely generated $KG$-module $V$ with $\bigO$-form $M_V$, we have the equality
$$\langle [P], [M_V/\mathfrak{p}M_V]\rangle_G = \langle [K\otimes_{\bigO}P],[V]\rangle_G,$$
which is precisely \cite[\S 15.4(b)]{serreRepsFiniteGroups}.  For profinite $G$ just take limits.
\end{proof}

\begin{corol}
Having naturally identified $\widehat{P}_{\overline{\bigO}}(-)$ with $\widehat{R}_{\overline{\F}_p}^*(-)$ and $\widehat{P}_{\overline{\Q}_p}(-)$ with $\widehat{R}_{\overline{\Q}_p}^*(-)$ as above, the natural transformations $\widehat{e}_{\overline{\Q}_p}$ and $\widehat{d}_{\overline{\Q}_p}^*$ are equal.  Thus $e_{\overline{\Q}_p}$ and $d_{\overline{\Q}_p}$ are dual.
\end{corol}

As with finite groups, we can identify projective $\bigO\db{G}$-modules by their shadow in characteristic 0:

\begin{prop}
The natural transformation $\widehat{e}_K$ is monic.  Thus, given two finitely generated projective $\bigO\db{G}$-modules $P,Q$, if $\widehat{e}_K([P])=\widehat{e}_K([Q])$, then $P\iso Q$.
\end{prop}

\begin{proof}
For finite groups this is \cite[\S 16.1 Corollary 2]{serreRepsFiniteGroups}.   That $\widehat{e}_K$ is monic now follows from the exactness of $\dirlim$ (see e.g. \cite[1.2.6]{ribzal}).
\end{proof}

\section{Characters}\label{SECTION characters}

Denote by $\overline{\Q}_p^\tn{dis}$ the field $\overline{\Q}_p$ given the discrete topology.  For a finite group $G$, denote by $\tn{Conj}(G)$ the set of conjugacy classes of $G$ and by $\tn{Conj}_{p'}(G)$ the set of conjugacy classes of $p$-regular elements of $G$.  The set of irreducible characters of $G$ over $\overline{\Q}_p$ forms a basis for the vector space of (continuous) class functions from $G$ to $\overline{\Q}_p^{\tn{dis}}$ (\cite[Theorem 6]{serreRepsFiniteGroups}).  Similarly, the set of irreducible Brauer characters of $G$ over $\overline{\F}_p$ forms a basis for the vector space of class functions from conjugacy classes of $p$-regular elements of $G$ (\cite[Theorem 42]{serreRepsFiniteGroups}).  Denoting the vector space of continuous functions from a space $X$ to $\overline{\Q}_p^\tn{dis}$ by $C(X,\overline{\Q}_p^\tn{dis})$, we can express this paragraph succinctly as
$$C(\tn{Conj}(G),\overline{\Q}_p^\tn{dis})\iso R_{\overline{\Q}_p}(G)\otimes_{\Z} \overline{\Q}_p$$
and
$$C(\tn{Conj}_{p'}(G),\overline{\Q}_p^\tn{dis})\iso R_{\overline{\F}_p}(G)\otimes_{\Z} \overline{\Q}_p.$$

Both $\tn{Conj}(-)$ and $\tn{Conj}_{p'}(-)$ are covariant functors from finite groups to topological spaces.  A standard compactness argument verifies that for a profinite group $G$ we have $\tn{Conj}(G)\iso \invlim_N \tn{Conj}(G/N)$.   By regarding an element $g$ of $G$ as $p$-regular if and only if it is $p$-regular in every continuous finite quotient (``the supernatural order of $\langle g\rangle$ is prime to $p$'' \cite[I.1.3]{SerreCohGal}), we also have $\tn{Conj}_{p'}(G)\iso \invlim_N \tn{Conj}_{p'}(G/N)$.  Thus we regard $\tn{Conj}(-)$ and $\tn{Conj}_{p'}(-)$ as functors from profinite groups to profinite topological spaces.

Now since $C(\tn{Conj}(-),\overline{\Q}_p^\tn{dis})$ and $C(\tn{Conj}_{p'}(-),\overline{\Q}_p^{\tn{dis}})$ are contravariant functors and direct limits commute with tensor products, we obtain the same description of class functions on a profinite group $G$ (for pro-$p$ groups cf. \cite[\S 11.4]{ArdWadArtinIrreps}):

\begin{prop}\label{character R isos}
For a profinite group $G$, we have ring isomorphisms
$$C(\tn{Conj}(G),\overline{\Q}_p^\tn{dis})\iso \widehat{R}_{\overline{\Q}_p}(G)\otimes_{\Z} \overline{\Q}_p$$
and
$$C(\tn{Conj}_{p'}(G),\overline{\Q}_p^\tn{dis})\iso \widehat{R}_{\overline{\F}_p}(G)\otimes_{\Z} \overline{\Q}_p.$$
\end{prop}

The dual statement requires some notation.  Given a finite set $X$, denote by $\overline{\Q}_p[X]$ the free $\overline{\Q}_p$-module with basis $X$ -- as a functor this is the composition of ``free abelian group'' with $(\overline{\Q}_p\otimes_{\Z}-)$.  If $X=\invlim X_i$ is a profinite space, we define $$\overline{\Q}_p\db{X}:=\invlim_i \overline{\Q}_p[X_i].$$
Given a covariant functor $F$ from profinite groups to topological abelian groups and a profinite group $G=\invlim G/N$, define
$$F(G)\,\ctens\, \overline{\Q}_p:=\invlim_N (F(G/N)\otimes_{\Z} \overline{\Q}_p).$$
For a finite group $G$, more-or-less by definition we have natural isomorphisms
$$\overline{\Q}_p[\textnormal{Conj}(G)]^*\iso C(\tn{Conj}(G),\overline{\Q}_p^\tn{dis})$$
and
$$\overline{\Q}_p[\textnormal{Conj}_{p'}(G)]^*\iso C(\tn{Conj}_{p'}(G),\overline{\Q}_p^\tn{dis})$$
and hence taking limits, the dual of Proposition \ref{character R isos} is

\begin{prop}
For a profinite group $G$, we have homeomorphisms
$$\overline{\Q}_p\db{\tn{Conj}(G)}\iso \widehat{P}_{\overline{\Q}_p}(G)\,\ctens\, \overline{\Q}_p$$
and
$$\overline{\Q}_p\db{\tn{Conj}_{p'}(G)}\iso \widehat{P}_{\overline{\bigO}}(G)\,\ctens\, \overline{\Q}_p.$$
\end{prop}

\section{Blocks}\label{SECTION blocks}

It follows from observations of Gabriel \cite{GabrielThesis} that block theory makes sense for profinite groups.  Recall that an idempotent $e$ in a ring $R$ is said to be \emph{central} if it lives in the center of $R$, \emph{orthogonal} to an idempotent $f$ if $ef=0$ and $fe=0$,  and (\emph{centrally}) \emph{primitive} if it cannot be written as the sum of two non-zero orthogonal (central) idempotents.  Fix a $p$-modular system $(K,\bigO,k)$ and a profinite group $G$.  The following proposition is due essentially to Gabriel.

\begin{prop}
Let $G$ be a profinite group.  There exists a set of pairwise orthogonal centrally primitive central idempotents $E=\{e_i\,|\,i\in I\}$ in $\bigO\db{G}$ such that
$$\bigO\db{G}\iso \prod_{i\in I} \bigO\db{G}e_i.$$
The images $E'=\{e_i'\,|\,i\in I\}$ of the $e_i$ in $k\db{G}$ are again pairwise orthogonal centrally primitive central idempotents, and
$$k\db{G}\iso \prod_{i\in I} k\db{G}e_i'.$$
\end{prop}

\begin{proof}
The first claim is \cite[IV\S3 Corollaries 1 and 2]{GabrielThesis}.  It is easily checked that the $e_i'$ are pairwise orthogonal central idempotents, and that $k\db{G}\iso \prod_{i\in I} k\db{G}e_i'$.  That the $e_i'$ are centrally primitive follows by noting that the lift of a central idempotent is unique (as one can see by mimicking the proof of \cite[Theorem 3.1(c)]{thevenaz}, for example).
\end{proof}

We call the indecomposable factors $\bigO\db{G}e_i, k\db{G}e_i'$ the \emph{blocks} of $\bigO\db{G}, k\db{G}$ respectively and the corresponding central idempotents the \emph{block idempotents}.  Blocks are not functorial in general, but are functorial if we restrict our interest to surjective group homomorphisms.  We return to finite groups, and denote by $\cat{SFGp}$ the category of finite groups and surjective group homomorphisms.  In order to make precise the connection between blocks, projectives and simples, we define the following functors from $\cat{SFGp}$ to abelian groups:

\medskip

For $A\in \{\bigO, k\}$, the covariant functor $\tn{Idem}_A(-)$ takes a finite group $G$ to the free abelian group with basis the conjugacy classes of primitive idempotents of $AG$.  Given a morphism $\alpha:G\onto H$,  $\tn{Idem}_A(\alpha)$ sends the idempotent $[e]\in \tn{Idem}_A(G)$ to $[\alpha(e)]\in \tn{Idem}_A(H)$.

The covariant functor $\tn{CIcov}_A(-)$ takes $G$ to the free abelian group with basis the primitive central idempotents of $AG$.  On morphisms,  $\tn{CIcov}_A(\alpha)$ sends the central idempotent $[c]\in \tn{CIcov}_A(G)$ to $[\alpha(c)]\in \tn{CIcov}_A(H)$.

The contravariant functor $\tn{CIcon}_A(-)$ takes $G$ to the free abelian group with basis the primitive central idempotents of $AG$.  On morphisms, $\tn{CIcon}_A(\alpha)$ sends the primitive central idempotent $[d]\in \tn{CIcon}_A(H)$ to $[c]$, where $c$ is the unique central primitive idempotent of $AG$ such that $d\alpha(c)=d$ (a standard argument shows that such a $c$ exists and is unique: existence follows since $\alpha(1)d=d$, while uniqueness follows from the fact that distinct primitive central idempotents are orthogonal.  See e.g. \cite[\S 37]{thevenaz} for more on this).

\medskip

In fact, $\tn{Idem}_A(-)$ makes sense as a functor with domain the category of finite groups and all homomorphisms, in which case it is naturally isomorphic to $P_A(-)$ by the correspondence $e\leftrightarrow AGe$.  However, as central idempotents need not be preserved by arbitrary maps, we require all group homomorphisms in this section to be surjective.  We regard $\tn{CIcov}_A(-)$ as a subfunctor of $\tn{Idem}_A(-)$ via the inclusion $[c]$ to $[c]$  --  this corresponds to taking the block $AGc$ to the projective left module $AGc$.  Define the pairing
$$\tn{CIcov}_A(G)\times \tn{CIcon}_A(G)\to \Z$$
on idempotents by setting $\langle c,d\rangle_G$ to be the number of non-zero primitive summands of $cd$, and extending bilinearly.  This pairing is non-singular and has the property that
$$\langle \tn{CIcov}(\alpha)(c),d\rangle_H =\langle c,\tn{CIcon}(\alpha)(d)\rangle_G,$$
so defines a duality
$$\tn{CIcov}_A(-)\iso \tn{CIcon}_A(-)^*.$$

Think of $R_k(-), \tn{CIcon}_k(-)$ as functors from $\cat{SFGp}$ and define a map $\eta_{k,G}:R_k(G)\to \tn{CIcon}_k(G)$ by sending the irreducible module $[V]$ to $\dim(V)[c]$, where $c$ is the unique central primitive idempotent such that $cV\neq 0$.

\begin{prop}
The map $\eta_{k,G}$ is the component at $G$ of a natural transformation $\eta_k: R_k(-)\to \tn{CIcon}_k(-)$.
\end{prop}

\begin{proof}
We are asserting that for $V$ an irreducible $kG$-module and $\alpha:G\onto H$, we have the equality $\dim({}^{\alpha}V)[d]=\dim(V)[c]$, where $c$ is the unique centrally primitive central idempotent such that $\alpha(c)V\neq 0$, and $d$ is the unique centrally primitive central idempotent in $kG$ with $d{}^{\alpha}V\neq 0$; that is, with $\alpha(d)V\neq 0$.  Since ${}^{\alpha}(-)$ preserves dimension, this is immediate.
\end{proof}

Taking direct limits over the finite extensions of $\F_p$, we get a natural transformation $\eta=\eta_{\overline{\F}_p}:R_{\overline{\F}_p}(-)\to \tn{CIcon}_{\overline{\F}_p}(-)$.

\begin{prop}
The natural transformation $\eta:  R_{\overline{\F}_p}(-)\to \tn{CIcon}_{\overline{\F}_p}(-)$ is dual to the inclusion $\tn{CIcov}_{\overline{\F}_p}(-)\to \tn{Idem}_{\overline{\F}_p}(-)$.
\end{prop}

\begin{proof}
For a finite group $G$, the square to be checked is
$$
\xymatrix{
 \tn{CIcov}_{\overline{\F}_p}(G)\ar[d]^{\sim}\ar[r]     & \tn{Idem}_{\overline{\F}_p}(G)\ar[d]^{\sim} \\
 \tn{CIcon}_{\overline{\F}_p}(G)^* \ar[r]^{\eta_{G}^*}                      & R_{\overline{\F}_p}(G)^*.}
$$
This amounts to saying that for a block idempotent $c$, the functions $\langle c,-\rangle\circ\eta_G$ and $\langle \overline{\F}_pGc,-\rangle$ are equal.  This is easily verified -- for a module $M$, both send $[M]$ to the dimension of $cM$.
\end{proof}

By functoriality, for $A\in\{\bigO,k\}$ we can thus define the following maps from the profinite group $G=\invlim_{N\lhd_O G}\,G/N$ to topological abelian groups:
$$\widehat{\tn{CIcov}}_A(G):=\invlim \tn{CIcov}_A(G/N),$$
$$\widehat{\tn{CIcon}}_A(G):=\dirlim \tn{CIcon}_A(G/N).$$
In this way we regard $\widehat{\tn{CIcov}}_A(-)$ and $\widehat{\tn{CIcon}}_A(-)$ as (covariant and contravariant, respectively) functors from profinite groups with surjective homomorphisms to topological abelian groups.  To obtain concrete descriptions of these functors, we note the following.

\begin{prop}
For $A\in\{\bigO,k\}$, restricting the surjective inverse system $\{A[G/N],\varphi_{MN}\}$ to the central idempotents of $A[G/N]$ expresses the central idempotents of $A\db{G}$ as the inverse limit of the central idempotents of the $A[G/N]$.

The set of centrally primitive central idempotents of $A\db{G}$ is a discrete subset of $A\db{G}$, and can be obtained as the direct limit of the centrally primitive central idempotents of the $A[G/N]$.
\end{prop}

\begin{proof}
The map from a profinite ring to itself defined by $x\mapsto x-x^2$ is continuous, so the inverse image of 0 is closed.  The first claim follows by applying this argument to the center of $A\db{G}$.

For the second claim, take $c$ a primitive central idempotent and find some map $\varphi$ to a finite quotient with $\varphi(c)\neq 0$.  The only primitive central idempotent in the inverse image of $\varphi(c)$ under $\varphi$ is $c$ itself, because $\varphi$ restricted to the primitive central idempotents not mapping to 0 is injective.  Hence, the set of primitive central idempotents is discrete.  The maps required for the direct system are those yielding $\tn{CIcon}(-)$: namely, given a quotient map $\varphi_{MN}:A[G/N]\to A[G/M]$, take the centrally primitive central idempotent $d$ in $A[G/M]$ to the unique centrally primitive central idempotent $c$ of $A[G/N]$ such that $\varphi_{MN}(c)d\neq0 $.
\end{proof}

It follows that the block idempotents of $A\db{G}$ are a basis for $\widehat{\tn{CIcon}}_A(G)$:

\begin{corol}\label{block isomorphisms}
For $A\in \{\bigO,k\}$ and $G$ a profinite group we have isomorphisms
$$\widehat{\tn{CIcon}}_A(G)\iso \bigoplus_{c}\Z_{[c]}$$
and
$$\widehat{\tn{CIcov}}_A(G)\iso \prod_{c}\Z_{[c]}$$
where $c$ runs through the set of centrally primitive central idempotents of $A\db{G}$.
\end{corol}

Taking limits of $\eta_{k,G/N}, \eta_{k,G/N}^*$ we obtain the natural transformation
$$\widehat{\eta}_k: \widehat{R}_k(-)\to \widehat{\tn{CIcon}}_k(-)$$
and its dual
$$\widehat{\eta}_k^*: \widehat{\tn{CIcov}}_k(-)\to \widehat{P}_{\bigO}(-).$$
Irreducible $k\db{G}$-modules are naturally partitioned into blocks by $\widehat{\eta}_{k,G}$.  Writing $\{c_i\,|\,i\in I\}$ for the maximal set of pairwise orthogonal centrally primitive central idempotents of $A\db{G}$ (for $A\in \{\bigO,k\}$), a projective $A\db{G}$-module $P$ splits into blocks by $P\mapsto \prod_i c_iP$, but note that this map does not commute with surjections.

One also obtains a natural ``$p$-block'' structure on $K\db{G}$, via the composition
$$
\xymatrix{
 \widehat{R}_K(-) \ar[r]^{\widehat{d}_{K}}   & \widehat{R}_{k}(-) \ar[r]^{\widehat{\eta}_{k}} & \widehat{\tn{CIcon}}_k(-).}
$$

\section{The Cartan homomorphism}\label{SECTION Cartan}

The Cartan homomorphism $c_{k,G}:P_k(G)\to R_k(G)$ for a finite group $G$ is simply given by $[P]\mapsto [P]$.  That is, it records the multiplicity of the composition factors of a projective $kG$-module $P$.  The information in $c_{k,G}$ is usually recorded in the \emph{Cartan Matrix} $C(G)$ of $G$ with respect to $k$.  The rows of $C(G)$ are indexed by simple $kG$-modules, the columns by indecomposable projective $kG$-modules, and $C(G)_{S,P}$ is the multiplicity of $S$ as a composition factor of $P$.  While the Cartan homomorphism is not a natural transformation, the information recorded in the Cartan matrix makes its computation a valuable tool in the study of finite groups.

We demonstrate in this section how we might try to understand projective modules for certain well behaved profinite groups in terms of the Cartan matrices of their finite quotients.  Our approach makes use of the Jennings filtration \cite{jennings}.

\medskip

Let $G$ be a finite group with $U$ a normal $p$-subgroup, and suppose $k$ is sufficiently large for $G$.  Let $J^0=kU, J^1=J=\rad(kU)$, and for each $i\in \N$ recursively define $J^i:=\rad(J^{i-1})$.

\begin{lemma}\label{Radical quotient isomorphism}
We have $kG/JkG\iso k[G/U]$.
\end{lemma}

\begin{proof}
We have $k[G/U]\iso (kG)_U\iso kG/I_UkG$, where $I_U$ is the kernel of the augmentation map $kU\to k$ given by $u\mapsto 1$.  The result now follows from the fact that $I_U=J$.
\end{proof}

Since $U$ is normal, one can show by induction that for each $d\in \N_0$, $J^i$ is a $kG$-module, with action from $G$ given by conjugating elements of $U$.  It follows that for each $i$, $J^i/J^{i+1}$ is a $kG$-module.  This also shows that $J^ikG$ is a submodule of $kG$, since for $g\in G,j\in J^i, x\in kG$ we have $gjx=(gjg^{-1})(gx)\in J^ikG$.

Define the map
$$\rho:J^i/J^{i+1}\otimes_k kG/JkG\to J^ikG/J^{i+1}kG$$
by $(j+J^{i+1}\otimes x)\mapsto jx+J^{i+1}kG$.  This map is a well-defined surjective $kG$-module homomorphism.

\begin{lemma}
The maps $J^i/J^{i+1}\otimes_k kG/JkG\to J^ikG/J^{i+1}kG$ are $kG$-module isomorphisms, for each $i$.
\end{lemma}

\begin{proof}
Since the maps are surjective, we need only show that the sum of the dimensions on either side agree.  The left hand side is
\begin{align*}
\sum_{i\geqslant 0}\dim_k(J^i/J^{i+1}\otimes_k kG/JkG)   &= \dim(kG/JkG)\sum_{i\geqslant 0}\dim (J^i/J^{i+1}) \\
                                                    &= \dim(k[G/U])\dim(kU) \\
                                                    &= (|G|/|U|)|U| = |G|.
\end{align*}
The right hand side is
\begin{align*}
\sum_{i\geqslant 0}\dim_k(J^ikG/J^{i+1}kG)               &= \sum_{i\geqslant 0}(\dim(J^ikG) - \dim(J^{i+1}kG)) \\
                                                    &= \dim J^0kG \\
                                                    &= \dim kG = |G|.
\end{align*}
\end{proof}

From this isomorphism and Lemma \ref{Radical quotient isomorphism} we observe that we have filtered $kG$ by projective $k[G/U]$-modules.  The same applies for arbitrary projectives:

\begin{lemma}\label{filtration isomorphism for P}
Let $P$ be a projective $kG$-module.  The maps $J^i/J^{i+1}\otimes_k P/JP\to J^iP/J^{i+1}P$ are isomorphisms, for each $i$.
\end{lemma}

Note that since $U$ acts trivially on simple $kG$-modules, the natural map $R_k(G/U)\to R_k(G)$ is an isomorphism.  Denote by $X$ the $kG$-module $kU$ given above, with action from $G$ given by conjugating elements of $U$.

\begin{prop}
Let $B$ be the matrix representing the map $R_k(G)\to R_k(G)$ given on simples by $[S]\mapsto [X\otimes_k S]$.  Then
$$C(G)=B\cdot C(G/U).$$
\end{prop}

\begin{proof}
The $(S,T)$th entry of $B$ is the number of times $S$ appears as a composition factor of $X\otimes_k T$.  The $(T,P_U)$th entry of $C(G/U)$ is the number of times $T$ appears as a composition factor of $P_U$.  Thus, the $(S,P_U)$th entry of $B\cdot C(G/U)$ is
$$\sum_{T}(\#\hbox{ times }S\hbox{ a factor of }X\otimes T)\cdot(\#\hbox{ times }T\hbox{ a factor of }P_U).$$
Meanwhile, the $(S,P)$th entry of $C(G)$ is the number of times $S$ appears as a composition factor of $P$.  But from Lemma \ref{filtration isomorphism for P} we have
$$[P]=\left(\sum_{i\geqslant 0}[J^i/J^{i+1}]\right)\otimes [P_U],$$
and hence (since $\{J^i/J^{i+1}\}$ is a filtration of $X$),
\begin{align*}
C(G)_{S,P}  & =\sum_T(\#\hbox{ times }T\hbox{ a factor of }P_U)\cdot(\#\hbox{ times }S\hbox{ a factor of }X\otimes T) \\
            & =(B\cdot C(G/U))_{S,P_U}.
\end{align*}
\end{proof}

We return now to infinite groups.  Let $G=\invlim_{i\in \N} G/U_i$ be a countably based virtually pro-$p$ group, with the $U_i$ a descending chain of open normal pro-$p$ subgroups.  Note that there are finitely many simple $k\db{G}$-modules since $G$ is virtually pro-$p$ -- indeed, we have $R_k(G)\iso R_k(G/U_i)$.  For each $i$, let $B_i$ be the matrix $R_k(G)\to R_k(G)$ given by tensoring with the finite $k\db{G}$-module $k[U_i/U_{i+1}]$ (with $G$ acting by conjugation).  Let $C_i$ be the Cartan matrix of $G/U_i$.  By induction on $i$ and the discussion above we see
$$C_n = B_nB_{n-1}\hdots B_2 C_1.$$
For an arbitrary countably based pro-$p$ group it seems difficult to control the matrices $B_i$.  We can understand the situation better when $G$ is a compact analytic pro-$p$ group.  Such a profinite group $G$ is characterized by the existence of an open normal uniformly powerful pro-$p$ subgroup $U$ (for further details about such groups, see \cite{dixonetal}).

\begin{theorem}\label{Bi matrices equal}
Let $G$ be a compact analytic pro-$p$ group.  Then we can find an inverse system of finite groups $G=\invlim_{i\in \N} G/U_i$ for which the matrices $B_i$ are equal.
\end{theorem}

\begin{proof}
Let $U$ be a uniformly powerful open subgroup of $G$ and write $U_0=G, U_1=U, U_2=[U,U]U^p, U_{i+1}=[U_i,U](U_i)^p$.  We have $G=\invlim_{i\in \N} G/U_i$.  For each $i\geqslant 1$ we have an isomorphism $\rho:U_i/U_{i+1}\to U_{i+1}/U_{i+2}$ given by powering by $p$.  This isomorphism extends to an isomorphism of $k\db{G}$-modules $k[U_i/U_{i+1}]\to k[U_{i+1}/U_{i+2}]$ since on a basis element $uU_{i+1}$ of $U_i/U_{i+1}$ we have
$$\rho(g\cdot uU_{i+1})=\rho(gug^{-1}U_{i+1})=gu^pg^{-1}U_{i+2}=g\cdot\rho(uU_{i+1}).$$
Thus the maps $(k[U_i/U_{i+1}]\otimes_k -)$ and $(k[U_{i+1}/U_{i+2}]\otimes_k -)$ on $R_k(G)$ are equal, and hence the matrices $B_i$ and $B_{i+1}$ are equal.
\end{proof}

\subsection{Example}

Let $k=\F_9$, a sufficiently large field for $G_n=\SL_2(\Z/3^n\Z)$. The Cartan matrices for $G_1, G_2, G_3$ over $k$ are
$C(G_1) = \left(\begin{array}{ccc}
3   & 0 & 0 \\
0   & 1 & 0 \\
0   & 0 & 3
\end{array}\right),$
$C(G_2) = \left(\begin{array}{ccc}
27  &18 & 0 \\
18  &21 & 0 \\
 0  & 0 &81
\end{array}\right)$
and
$C(G_3) = \left(\begin{array}{ccc}
567 &540& 0 \\
540 &549& 0 \\
0   & 0 &2187
\end{array}\right)$.

Let $G=\SL_2(\Z_3)=\invlim_n G_n$.  From \cite[Theorem 5.2]{dixonetal} we know that the kernel $U_1$ of the natural map $G\onto G_1$ is an open normal uniformly powerful pro-$3$ subgroup of $G$, and we can take for $U_i$ the kernel of $G\onto G_i$.  Thus by Theorem \ref{Bi matrices equal} the matrices $B_i$ discussed above are equal.  From $C(G_1)$ and $C(G_2)$ we deduce that
$$B=C(G_2)C(G_1)^{-1}=
\left(\begin{array}{ccc}
9   &18 & 0 \\
6   &21 & 0 \\
0   & 0 & 27
\end{array}\right).$$

By diagonalizing $B$ we can now write down the general formula for $C(G_n) = B^{n-1}C(G_1)$:
$$C(G_n)=\frac{1}{4}
\left(\begin{array}{ccc}
3^{3n-2}+3^{n+1}    & 3^{3n-2}-3^{n}    & 0                 \\
3^{3n-2}-3^{n}      & 3^{3n-2}+3^{n-1}  & 0                 \\
0                   & 0                 & 4\cdot 3^{3n-2}
\end{array}\right).$$
The determinant of $C(G_n)$ is $3^{7n-5}$.

\section{Acknowledgements}

A big thanks to Jeremy Rickard for helpful conversations.

\end{document}